\documentclass[a4paper,reqno,11pt]{amsart}

\usepackage{geometry}
\usepackage{amsfonts,amsmath,amssymb}
\usepackage{hyperref}
\hypersetup{
   colorlinks,
    citecolor=blue,
    filecolor=black,
    linkcolor=blue,
    urlcolor=magenta
}

\newtheorem{theorem}{Theorem}[section]
\newtheorem{lemma}[theorem]{Lemma}

\theoremstyle{definition}

\newtheorem{example}[theorem]{Example}
\newtheorem{remark}[theorem]{Remark}
\newcommand{\norm}[1]{\left\Vert#1\right\Vert}

\allowdisplaybreaks
\numberwithin{equation}{section}

\begin{document}
\title[On a nonhomogeneous heat equation on the complex plane]{On a nonhomogeneous heat equation on the complex plane}

\author{Duong Ngoc Son}
\address{Faculty of Fundamental Sciences,
Phenikaa University, Hanoi, 12116
Vietnam}
\email{son.duongngoc@phenikaa-uni.edu.vn}

\author{Tran Van Thuy}
\address{Faculty of Fundamental Sciences, East Asia University of Technology, Trinh Van Bo, Nam Tu Liem, Hanoi, Vietnam}
\email{thuytv@eaut.edu.vn}

\author{Pham Truong Xuan}
\address{Thang Long Institute of Mathematics and Applied Sciences (TIMAS), Thang Long University,
Nghiem Xuan Yem, Hoang Mai, Hanoi, Vietnam.} 
\email{xuanpt@thanglong.edu.vn} 

\date{July 4, 2024}

\begin{abstract}  
In this article, we investigate the existence, uniqueness, and asymptotic behaviors of mild solutions of a parabolic evolution equations on complex plane, in which the diffusion operator has the form \(\overline{\Box}_{\varphi} = \overline{D}\, \overline{D}^{\ast}\), where \(\overline{D} f = \bar{\partial}f + \varphi_{\bar{z}} f\), the function \(\varphi\) is smooth and subharmonic on \(\mathbb{C}\), and \(\overline{D}^{\ast}\) is the formal adjoint of \(\overline{D}\). Our method combines certain estimates of heat kernel associating with the homogeneous linear equation of Raich \cite{raich06} and a fixed point argument. 
\end{abstract}

\maketitle

\tableofcontents

\section{Introduction}

In an interesting paper \cite{christ}, Christ studied the differential operator on \(\mathbb{C}\)
\[ \overline{D} f = \frac{\partial}{\partial \bar{z}} + \frac{\partial \varphi}{\partial \bar{z}} f, \]
assuming that \(\varphi\) is a subharmonic and non-harmonic function. This operator is closely related to the (weighted) Cauchy--Riemann operator \(\mathbb{C}\). For example, it is immediate that \(\overline{D} f = 0\) if and only if \(e^{\varphi}f\) is holomorphic. In fact, we have the following relation
\[ 
\overline{D} = e^{-\varphi} \frac{\partial}{\partial \bar{z}} e^{\varphi}.
\]

The motivation for Christ's work comes from at least two sources: the weighted \(\bar{\partial}\)-equation on \(\mathbb{C}\) (i.e., solving the nonhomogeneous equation \(\bar{\partial}f = g\) in \(L^2(\mathbb{C}, e^{-2\varphi})\), see, e.g., Berndtsson \cite{berndtsson}) and the tangential \(\bar{\partial}_b\)-equation on the rigid real hypersurface in $\mathbb{C}^2$ 
\[
M _{\varphi} = \{(z,w) \in \mathbb{C}^2 \mid \mathrm{Im}(w) = \varphi(z)\}.
\]
(via a partial Fourier transform, see, e.g., Christ \cite{christ} and Raich \cite{raich06}). Both problems are important in complex analysis and CR geometry and have been studied extensively, see, for instances, H\"ormander \cite{hormander}, Christ \cite{christ}, Raich \cite{raich06}, Haslinger \cite{haslinger}, and the references therein. In particular, Christ's work provides a number of interesting results concerning the equation \(\overline{D} f = g\), especially when the Laplacian \(\Delta \varphi\) is nonnegative but vanishes at some point and, as a measure, satisfies a volume doubling condition, see \cite{christ} for the details.

Our main motivation for this paper comes from an interesting paper of Raich \cite{raich06} who studied the heat equation whose diffusion term is constructed from \(\overline{D}\). Precisely, Raich consider the equation 
 \begin{equation}\label{e:heateq}
 	\begin{cases}
 		\dfrac{\partial u}{\partial t} + \overline{\Box}_{\varphi} u & = 0, \\
 		u(0,z) & = u_0(z), \quad z\in \mathbb{C},
 	\end{cases}
 \end{equation}
where \(\overline{\Box}_{\varphi} u = \overline{D}\, \overline{D}^{\ast} u\) and  \(\overline{D}^{\ast}\) is the formal self-adjoint of \(\overline{D}\) with respect to the Lebesgue measure on \(\mathbb{C}\). In fact, Raich studied a family of operators \(\overline{D} = \overline{D}_{\tau p}\) by taking \(\varphi = \tau p\), where \(p\) is a fixed subharmonic and nonharmonic function on \(\mathbb{C}\), and \(\tau\) is a positive real parameter. Raich proved a number of fundamental results related to this equation, its heat kernel, and its heat semigroup. Note that the dependence of Raich's estimates in \(\tau\) is important for several aspects, such as for transferring the results from \(\mathbb{C}\) to the real hypersurface \(M_\varphi\).

In this paper, following above mentioned works, we initiate a study of the following parabolic evolution equation
\begin{equation}\label{e:heateq0}
	\begin{cases}
		\dfrac{\partial u}{\partial t} + \overline{\Box}_{\varphi} u & = f(u), \\
		u(0,z) & = u_0(z), \quad z\in \mathbb{C},
	\end{cases}
\end{equation}
where $f(t,u)=|u|^{m-1}u$ and $m>2$. Note that, we can obtain similar results for a more general nonlinearity $f(t,u)$ satisfying 
\begin{equation}\label{e:nonlinearterm}
|f(t,u_1)-f(t,u_2)| \leqslant L|u_1-u_2|\left(  |u_1|^{m-1}+|u_2|^{m-1}\right),  \text{   where   }  L>0.
\end{equation}
{We briefly mention that semilinear equations similar to \eqref{e:heateq0}, which is a model for a reaction-diffusion process, have been studied for a long time especially when the underlying space is \(\mathbb{R}^n\) or a Riemannian manifold. They have important applications to geometry (e.g., to geometric flows on the underlying manifold) and to other areas; see, e.g., Weissler \cite{We1980, We1981}. Therefore, it is desirable to settle some basic questions regarding equation \eqref{e:heateq0} in our complex setting}. In fact, our main aim is to establish the well-posedness of the equation and study the asymptotic behavior of its mild solutions.

{As usual, we denote the space of all continuous functions from $\mathbb{R}_+$ to $L^{m-1}(\mathbb{C})$ by $BC(\mathbb{R}_+, L^{m-1}(\mathbb{C}))$ which is a Banach space equipped with supremum norm, that is 
\[ 
	\|v\|_{BC(\mathbb{R}_+,L^{m-1}(\mathbb{C}))} = \sup_{t \geq 0} \|v(t)\|_{L^{m-1}(\mathbb{C})}.
 \]}To explain our first result, we introduce, for $1<m-1<q<m(m-1)$,  the following Banach space

\begin{equation}\label{e:banachspace}
Y=\left\{v\in BC(\mathbb{R}_+,L^{m-1}(\mathbb{C})): \norm{v}_{BC(\mathbb{R}_+,L^{m-1}(\mathbb{C}))} + \sup_{t>0}t^{\frac{1}{m-1}-\frac{1}{q}}\norm{v(t)}_{L^q(\mathbb{C})} <+\infty  \right\},
\end{equation}
endowed with the norm 
\begin{equation}\label{e:norm}
\norm{v}_Y = \norm{v}_{BC(\mathbb{R}_+,L^{m-1}(\mathbb{C}))} + \sup\limits_{t>0}t^{\frac{1}{m-1}-\frac{1}{q}}\norm{v(t)}_{L^q(\mathbb{C})}.
\end{equation}
Our first result is the well-posedness of the semilinear equation and \eqref{e:heateq0} in~\(Y\), which can be stated as follows.
\begin{theorem}[Well-posedness and polynomial stability]\label{case1}
Let \(\varphi\) be a smooth subharmonic function on \(\mathbb{C}\). Let $1<m-1<q<m(m-1)$ and $u_0\in L^{m-1}(\mathbb{C})$.	The following assertions hold:
	\begin{enumerate}
		\renewcommand{\labelenumi}{(\roman{enumi})}
		
		\item
		If $\norm{u_0}_{L^{m-1}(\mathbb{C})}$ and $\rho>0$ are small enough, then the semilinear equation \eqref{e:heateq0} has a unique mild solution $\hat{u}$ in the ball $\mathbb\mathrm{B}(0,\rho)=\left\{ v\in Y : \norm{v}_{Y} \leqslant \rho  \right\}.$
		
		\item
		The above solution $\hat u$ is \emph{polynomial stable} in the sense that for any other solution $u \in BC(\mathbb{R}_+,L^{q}(\mathbb{C}))$ of \eqref{e:heateq0} with the initial data such that $\norm{u(0)-\hat{u}(0)}_{L^{m-1}(\mathbb{C})}$ is small enough, then we have
		\begin{equation}\label{e:pstable}
			\norm{u(t)-\hat{u}(t)}_{L^{q}(\mathbb{C})}  \leqslant {D} t^{-\left(\frac{1}{m-1}-\frac{1}{q}\right)}\norm{u(0)-\hat{u}(0)}_{L^{m-1}(\mathbb{C})},\quad \forall t>0.
		\end{equation}
	\end{enumerate}
\end{theorem}
In essence, Theorem~\ref{case1} establishes a well-posedness and a polynomial stability estimate for the semilinear equation \eqref{e:heateq0} for a general subharmonic function \(\varphi\). When \(\varphi\) statisfies a stronger ``convexity'' condition, we obtain, in the second result, a better stability estimate. To state this, we need to recall some geometric objects originated from 
Carnot--Carath\'eodory geometry developed by Nagel--Stein--Wainger \cite{nsw}, as presented in Raich \cite{raich06}. Namely, for a fixed subharmonic (nonharmonic) function \(\varphi\), \(z\in \mathbb{C}\), and \(r > 0\), we write
\begin{equation}\label{key}
	a_{jk}^z = \frac{1}{j!k!}\frac{\partial^{j+k} \varphi}{\partial z^j \partial \bar{z}^k}(z).
\end{equation}
Let
\begin{equation}\label{e:mu}
	\mu(z,r) = \inf_{j,k \geq 1 }\left|\frac{r}{a^z_{jk}}\right|^{1/(j+k)},
\end{equation}
with the convention that \(|r/a^{z}_{jk}| = + \infty\) if \(a^{z}_{jk} = 0\). 
For example, when \(\varphi(z) = |z|^2\), we easily have 
\(\mu(z,r) = \sqrt{r}\). Finally, we define the following quantity
\begin{equation}\label{e:delta}
	\delta(\varphi) = \inf_{\mathbb{C}} \mu(z,1)^{-2} \geq 0.
\end{equation}
Our second result can be stated as follows:
\begin{theorem}[Well-posedness and exponential stability]\label{case2}
Let \(\varphi\) be a subharmonic non-harmonic \emph{polynomial} of \(z\). Then, for a given constant $n>m-1>1$, the following assertions hold:
\begin{enumerate}
\renewcommand{\labelenumi}{(\roman{enumi})}
 
\item
If $\norm{u_0}_{L^{n}(\mathbb{C})}$ and $\rho>0$ are small enough, then the heat equation \eqref{e:heateq0} has a unique mild solution $\hat{u}$ in the small ball $B(0,\rho)$ of the Banach space $BC(\mathbb{R}_+,L^{n}(\mathbb{C}))$.

\item
The above solution $\hat u$ is \emph{exponentially stable} in the sense that for any other solution $u \in BC(\mathbb{R}_+,L^{n}(\mathbb{C}))$ of \eqref{e:heateq0} such that $\norm{u(0)-\hat{u}(0)}_{L^{n}(\mathbb{C})}$ is small enough, then we have
\begin{equation}\label{e:estable}
\norm{u(t)-\hat{u}(t)}_{L^{n}(\mathbb{C})}  \leqslant {\widetilde{D}} e^{-\sigma t}\norm{u(0)-\hat{u}(0)}_{L^n(\mathbb{C})},\quad \forall t>0,
\end{equation}
where $0<\sigma<\widetilde{C}\delta(\varphi)$ and $\delta(\varphi) > 0$.
\end{enumerate}
\end{theorem}

{As briefly mentioned before,} semilinear equations similar to \eqref{e:heateq0} have been studied extensively in various situations including those related to the Laplace--Beltrami operator on Riemannian manifolds, the drifted Laplacian on smooth metric measure spaces, as well as sub-elliptic Laplacian on sub-Riemannian manifolds (e.g., on Heisenberg group), see, e.g., Oka \cite{Oka2018}. However, less was known in our current setting of the \(\overline{\Box}_{\varphi}\) operator on the complex plane. On the other hand, one can consider \eqref{e:heateq0} for  \(\overline{\Box}_{\varphi}\) in the higher dimensional case. In this case, the assumption imposed on the weight function \(\varphi\) should be stronger and often related to the eigenvalues of its complex Hessian; see, e.g., Haslinger \cite{haslinger} and Dall'Ara \cite{dall} where such a condition plays an important role in their studies.  We hope to return to this case in the future.

The rest of this article is organized as follows. In Section~\ref{S2}, we briefly recall several required materials, mostly from Raich \cite{raich06}. We also give a $L^p$--$L^q$ estimate for the heat semigroup, with a standard proof. In Section~\ref{S3}, we give proofs of two main theorems. Finally, in Section~\ref{S4}, we construct two examples to illustrate two different cases appearing in our main results.

\section{Preliminaries}\label{S2}

\subsection{\(\overline{D}\)-operator and the heat equation on \(\mathbb{C}\)}
Let \(\varphi\) be a smooth subharmonic function on \(\mathbb{C}\). Consider the differential operator
\begin{equation}\label{e:dba}
\overline{D} f = \frac{\partial f}{\partial \bar{z}} + \frac{\partial \varphi}{\partial \bar{z}} f.
\end{equation}
By the usual construction, the Laplacian associated to \(\overline{D}\) is then given by
\begin{equation}\label{e:laplacian}
	\overline{\Box}_{\varphi} u = \overline{D}\, \overline{D}^{\ast} u,
\end{equation}
where 
\[
\overline{D}^{\ast}f = -\frac{\partial f}{\partial z} - \frac{\partial \varphi}{\partial z} f
\]
is the formal adjoint of \(\overline{D}\) in $L^2(\mathbb{C}, dA)$, $dA$ being the usual Lebesgue measure on the complex plane. Explicitly,
\begin{equation}\label{key}
\overline{\Box}_{\varphi} u = - \frac{\partial^2 u}{\partial z \partial \bar{z}} - \frac{\partial \varphi}{\partial \bar{z}} \frac{\partial u}{\partial z} + \frac{\partial \varphi}{\partial z} \frac{\partial u}{\partial \bar{z}} + \left(\left|\frac{\partial \varphi}{\partial \bar{z}}\right|^2 + \frac{\partial^2 \varphi}{\partial z\partial \bar{z}}\right) u.
\end{equation}
This operator turns out to be a Schr\"odinger operator with magnetic field and an electric potential, which has been studied in several interesting papers, see, e.g., Haslinger \cite{haslinger2}, Heffner--Haslinger \cite{hh}.  It is essentially self-adjoint on \(C^{\infty}_0(\mathbb{C})\) whose spectral theory has been quite well understood. Generally speaking, spectral properties of \(\overline{\Box}_{\varphi}\) heavily depend on the ``size'' of \(\Delta \varphi\) when \(|z| \to \infty\). An approach to the spectral theory for \(\overline{\Box}_{\varphi}\) is via the \(\bar{\partial}\)-Neumann problem. We also point out that, in fact, \(\overline{\Box}_{\varphi}\) is unitary equivalent to the weighted complex Laplacian on \(L^2(\mathbb{C}, e^{-2\varphi})\), cf. Haslinger \cite{haslinger}

In \cite{raich06}, Raich studied the heat equation \eqref{e:heateq} associated to the Laplacian \eqref{e:laplacian} and established fundamental results regarding this equation, its heat kernel estimates, and its associated heat semigroup. In particular, it follows that if \(u_0\) vanishes as \(|z| \to \infty\), then one has the following formula for the mild solution to the heat equation \eqref{e:heateq} as well as the heat semigroup: At \(z \in \mathbb{C}\) fixed,
\begin{equation}\label{e:sol}
	u(t,z) 
	= 
	e^{-t\, \overline{\Box}_{\varphi}} u_0(z)
	=
	\int_{\mathbb{C}} H(t,z,w) u_0(w) dA(w),
\end{equation}
where \(H(t,z,w)\) is the distribution heat kernel.
Raich \cite{raich06} established the regularity of the heat kernel and obtained strong estimates for it. Precisely, Raich's heat kernel estimate, specializing to the case \(\tau = 1\), is read as follows.
\begin{theorem}[Raich \cite{raich06}]\label{thm:2.1} Let \(\varphi\) be a subharmonic, nonharmonic polynomial and let \(\mu(z,r)\) be defined as in \eqref{e:mu}. If \(n\geq 0\) and \(Y^{\alpha}\) is the product of \(|\alpha|\) operators \(\overline{D}\) or \(\overline{D}^{\ast}\). Then
\begin{equation}\label{eq:heatest}
	\left|\frac{\partial^n}{\partial t^n} Y^{\alpha} H(t,z,w)\right| \leq C  t^{-1-n - |\alpha|/2} \exp\left(-\frac{|z-w|^2}{32t} - C't (\mu(z,1)^{-2} + \mu(w,1)^{-2})\right),
\end{equation}
where \(C , C' > 0\), and \(C'\) can be taken with no dependence on \(n\) and \(\alpha\).
\end{theorem}
We refer to Raich \cite{raich06} for a general version of the estimate where \(\varphi = \tau p\) and whose involving constants also depend on a parameter \(\tau>0\). The heat kernel estimate in Theorem~\ref{thm:2.1} will be of fundamental importance for us.

\subsection{\(L^p\)--\(L^q\) estimates for the heat semigroup} From the heat kernel estimate of Raich (Theorem~\ref{thm:2.1}), we can easily prove the following lemma.

\begin{lemma} Let \(\varphi\) be a smooth subharmonic function on \(\mathbb{C}\). Assume that 
\begin{equation}\label{e:delta}
	\delta = \delta (\varphi) = \inf_{z \in \mathbb{C}} \mu(z,1)^{-2} \geq 0,
\end{equation}
where $\mu$ is defined in \eqref{e:mu}. {Assume that \(\varphi\) is a polynomial}, then there exist positive constants \(C\) and \(\widetilde{C}\), such that for \(1 \leq q \leq p \leq  \infty\), it holds that 
\begin{equation}\label{e:lp-lq}
	\left\| e^{-t\, \overline{\Box}_{\varphi}} \psi \right\|_{L^p} \leq C t^{-\left(\frac{1}{q} - \frac{1}{p}\right)} e^{ - \widetilde{C}\delta t}\, \left\|\psi \right\|_{L^q}.
\end{equation}
{In the general case, i.e.,  $\varphi$ is not assumed to be a polynomial, then \eqref{e:lp-lq} holds with \(\widetilde{C} = 0\).}
\end{lemma}
\begin{proof}  The argument below is standard but we include the details for completeness. First, {assume that \(\varphi\) is a polynomial}. We put
 \[ \psi_t(z) = \left(e^{-t\, \overline{\Box}_{\varphi}} \psi\right)(z), \quad N_t(z) =  \exp\left(-\frac{|z|^2}{32t}\right).\]
From the heat kernel estimate \eqref{eq:heatest} (with $\widetilde{C} = 2C'$), we have
\begin{align*}
	\left|\psi_t(z) \right|
	& = \left|\int_{\mathbb{C}} H(t,z,w) \psi (w) dA(w)\right| \\
	& \leq C t^{-1} e^{-\widetilde{C} \delta t} \int_{\mathbb{C}}\exp\left(-\frac{|z-w|^2}{32t}\right)|\psi(w)| dA(w) \\
	& = C t^{-1} e^{-\widetilde{C} \delta t} \left( N_t \ast |\psi|\right).
\end{align*}
Using Young's inequality, we obtain
\begin{align}\label{e:young}
\|\psi_t\|_{L^p} \leq C  t^{-1} e^{-\widetilde{C} \delta t} \|N_t\|_{r} \, \|\psi\|_{L^q},
\end{align}
where 
\(1 + \frac{1}{p} = \frac{1}{r} + \frac{1}{q}\).  To complete the proof, we notice that 
\begin{align*}
\|N_t\|_{r} 
= \left(\int_{\mathbb{C}} \exp\left(-\frac{r|z|^2}{32t}\right)dA(z) \right) ^{1/r}
=
 \left(\frac{32t\pi}{r} \right) ^{1/r}.
\end{align*}
Plugging this into \eqref{e:young} above, we obtain
\begin{equation}\label{key}
	\|\psi_t\|_{L^p} \leq C  t^{-\left(\frac{1}{q} - \frac{1}{p}\right)} e^{-\widetilde{C} \delta t}\, \|\psi\|_{L^q},
\end{equation}
as desired.

{If $\varphi$ is not assumed to be a polynomial, then we have a weaker estimate for the heat kernel. Precisely, it is proved in Raich \cite[Theorem 25]{raich06} that, when \(V: = \Delta \varphi \geq 0\),
\begin{equation}\label{e:hegen}
	|H(t,z,w)| \leq \frac{1}{\pi t} \exp\left({-\frac{|z-w|^2}{t}}\right).
\end{equation}
From this estimate, we can also prove \eqref{e:lp-lq} for \(\widetilde{C} = 0\). The details are left to the readers.} The proof is complete.
\end{proof}
\begin{remark}
\begin{enumerate}
\renewcommand{\labelenumi}{(\roman{enumi})}

\item
It is worth pointing out that when \(\varphi = 0\), the operator \(\overline{\Box}_{\varphi}\) is (up to a constant) the Laplace operator on \(\mathbb{C} \simeq \mathbb{R}^2\) for which the \(L^p\)--\(L^q\) estimate above is a special case of the well-known \(L^p\)--\(L^q\) estimates for the heat semigroup on \(\mathbb{R}^n\). More generally, when \(\delta(\varphi) = 0\) (for example, when \(\varphi\) is harmonic), our estimate is very similar to its \(\mathbb{R}^n\) counterpart as well as in the setting of Heisenberg group $\mathbb{H}^n$, see Oka \cite{Oka2018}. The interesting case is, of course, when \(\varphi\) is non-harmonic. Note that, the $L^p$--$L^q$ estimate for the heat operator associated to the sub-Laplace operator in Oka \cite{Oka2018} was established for Heisenberg group of real dimension $n\geqslant 3$, meanwhile our estimate is valid on two dimension since $\mathbb{C}\simeq \mathbb{R}^2$.  

\item
It is not unexpected that when \(\delta(\varphi) > 0\) (e.g., when \(\varphi(z) = |z|^2\)) our estimate is stronger than its real counterpart. In particular, in this case we get the $L^p$--$L^q$ estimates with \emph{exponential} decay rates.
\end{enumerate}
\end{remark}

\section{Well-posedness and asymptotic behavior: Proofs}\label{S3}
In this section, we prove the global-in-time well-posedness of the equation \eqref{e:heateq0} and the asymptotic behavior of its mild solutions.
Let us recall that the mild solution to \eqref{e:heateq0} is defined as the solution of the integral equation
\begin{equation}\label{e:key}
		u(t) 
		= e^{-t\overline{\Box}_\varphi} u_0 + \int_0^t e^{-(t-s)\overline{\Box}_\varphi}f(s,u(s)) ds.
\end{equation}

The local and global well-posedness of heat equations (with the same nonlinearity as the right hand side of equation \eqref{e:heateq0}) in Euclidean space $\mathbb{R}^n$ were done by Weissler in \cite{We1980} and \cite{We1981}, respectively. It might be possible to use the same methods in these works to establish the well-posedness results. However, our approach is to pass from a ``linearized'' integral equation to the original \eqref{e:heateq0} via a fixed point argument.

In the two proofs below, we shall use the following functions:
\begin{enumerate}
\renewcommand{\labelenumi}{(\roman{enumi})}
	\item
	The usual Gamma function \(\Gamma\) is defined by $\Gamma(\theta) = \int_0^\infty s^{\theta-1}e^{-s}ds$, which is finite for $0<\theta<1$.
	
	\item
	The Beta function is defined by  $\mathrm{B}(k,l)= \int_0^1(1-s)^{k-1}s^{l-1}ds$, which is finite for
	all $k>0$ and $l>0$. Observe that for $k,\, l <1$ and $t>0$, the change of variable $s\to st$ yields
	\begin{align}\label{beta}
		t^{k+l-1}\int_0^t(t-s)^{-k}s^{-l}ds &= t^{k+l-1}t^{1-k-l}\int_0^1(1-s)^{-k}s^{-l}ds \notag\\
		&= \mathrm{B}(1-k,1-l)<\infty.
	\end{align}
\end{enumerate}
\subsection{Proof of Theorem~\ref{case1}}

First, for each \(v \in Y\), we consider the following integral \begin{equation}\label{le:weak}
u(t) = e^{-t \overline{\Box}_{\varphi} } u_0 + \int_0^t e^{-(t-s) \overline{\Box}_{\varphi} }f(s,v(s)) ds.
\end{equation}
We shall prove the boundedness of right hand-side (RHS) of \eqref{le:weak} in with the norm $\norm{\cdot}_Y$. From the $L^p$--$L^q$ estimates \eqref{e:lp-lq} and H\"older's inequality we have
\begin{align*}
\|\hbox{RHS of } & \hbox{\eqref{le:weak}}\|_{L^q(\mathbb{C})} \leq  
\norm{e^{-t\overline{\Box}_\varphi}u_0}_{L^q(\mathbb{C})} + \norm{\int_0^t e^{-(t-s)\overline{\Box}_\varphi}f(s,v(s)) ds}_{L^q(\mathbb{C})}\cr
&\leq  C  t^{-\left(\frac{1}{m-1}-\frac{1}{q} \right)}\norm{u_0}_{L^{m-1}(\mathbb{C})} + \int_0^t \norm{e^{-(t-s)\overline{\Box}_\varphi}f(s,v(s))}_{L^q(\mathbb{C})} ds\cr
&\leq  C  t^{-\left(\frac{1}{m-1}-\frac{1}{q} \right)}\norm{u_0}_{L^{m-1}(\mathbb{C})} + C  \int_0^t (t-s)^{-\left(\frac{m}{q}-\frac{1}{q}\right)} \norm{f(s,v(s))}_{L^{\frac{q}{m}}(\mathbb{C})} ds\cr
&\leq  C  t^{-\left(\frac{1}{m-1}-\frac{1}{q} \right)}\norm{u_0}_{L^{m-1}(\mathbb{C})} + C  \int_0^t (t-s)^{-\left(\frac{m}{q}-\frac{1}{q}\right)} \norm{v(s)}^m_{L^{q}(\mathbb{C})} ds\\
&=  C  t^{-\left(\frac{1}{m-1}-\frac{1}{q} \right)}\norm{u_0}_{L^{m-1}(\mathbb{C})} + C  \int_0^t (t-s)^{-\frac{m-1}{q}} \norm{v(s)}^m_{L^{q}(\mathbb{C})} ds,
\end{align*}
where the constant \(C \) does not depend on \(q\) or \(m\). 

Set $\omega(t) = t^{\frac{1}{m-1}-\frac{1}{q}}\norm{v(t)}_{L^q(\mathbb{C})}$. 
Multiplying both sides of above inequality with $t^{\frac{1}{m-1}-\frac{1}{q}}$, we get
$t^{\frac{1}{m-1}-\frac{1}{q}}$, we get
\begin{align}\label{core}
t^{\frac{1}{m-1}-\frac{1}{q}}\norm{u}_{L^q(\mathbb{C})} &\leq  C  \norm{u_0}_{L^{m-1}} + C  t^{\frac{1}{m-1}-\frac{1}{q}}\int_0^t (t-s)^{-\frac{m-1}{q}}s^{-\left( \frac{m}{m-1}-\frac{m}{q}\right)} (\omega(s))^m ds \notag \\
&\leq  C  \norm{u_0}_{L^{m-1}(\mathbb{C})} + C  (\sup_{t>0}\omega(t))^m t^{\frac{1}{m-1}-\frac{1}{q}} \int_0^t (t-s)^{-\frac{m-1}{q}}s^{-\left( \frac{m}{m-1}-\frac{m}{q}\right)} ds \notag \\
&\leq  {C  \norm{u_0}_{L^{m-1}(\mathbb{C})} + C \mathrm{B}\left( 1 - \gamma - \nu, 1 - m \nu\right)\norm{v}_Y^m \hbox{   (we used \eqref{beta})}}\notag\\
&\leq  {C  \norm{u_0}_{L^{m-1}(\mathbb{C})} + C \mathrm{B}\left( 1 - \gamma - \nu, 1 - m \nu\right)\norm{v}_Y^m }\notag\\
&\leq  C  \norm{u_0}_{L^{m-1}(\mathbb{C})} + \widehat{C}\norm{v}_Y^m,
\end{align}
where \(0 < \gamma = \frac{m}{q} - \frac{1}{m-1} < 1\) and \(\nu = \frac{1}{m-1} - \frac{1}{q} > 0 \), provided {that} $1<m-1<q<m(m-1)$. {Note that \(0 < \gamma + \nu <1\) and \(0 < m \nu <1\)}. Moreover, we have also that
\begin{align}\label{BoundLm}
\norm{u(t)}&_{L^{m-1}(\mathbb{C})} \leq 
\norm{e^{-t\overline{\Box}_\varphi}u_0}_{L^{m-1}(\mathbb{C})} + \norm{\int_0^t e^{-(t-s)\overline{\Box}_\varphi}f(s,v(s)) ds}_{L^{m-1}(\mathbb{C})}\cr
&\leq  C  \norm{u_0}_{L^{m-1}(\mathbb{C})} + \int_0^t \norm{e^{-(t-s)\overline{\Box}_\varphi}f(s,v(s))}_{L^{m-1}(\mathbb{C})} ds\cr
&\leq  C  \norm{u_0}_{L^{m-1}(\mathbb{C})} + C  \int_0^t (t-s)^{-\left(\frac{m}{q}-\frac{1}{m-1}\right)} \norm{f(s,v(s))}_{L^{\frac{q}{m}}(\mathbb{C})} ds\cr
&\leq  C \norm{u_0}_{L^{m-1}(\mathbb{C})} + C  \int_0^t (t-s)^{-\left(\frac{m}{q}-\frac{1}{m-1}\right)} \norm{v(s)}^m_{L^{q}(\mathbb{C})} ds\cr
&\leq  C  \norm{u_0}_{L^{m-1}(\mathbb{C})} + C  \int_0^t (t-s)^{-\left(\frac{m}{q}- \frac{1}{m-1}\right)} s^{-\frac{m}{m-1}+\frac{m}{q}}\left(s^{\frac{1}{m-1}-\frac{1}{q}}\norm{v(s)}_{L^{q}(\mathbb{C})}\right)^m ds\cr
&\leq  C  \norm{u_0}_{L^{m-1}(\mathbb{C})} + C  \left(\sup_{t>0}t^{\frac{1}{m-1}-\frac{1}{q}}\norm{v(t)}_{L^{q}(\mathbb{C})}\right)^m \times \notag \\
& \qquad \times \int_0^t (t-s)^{-\left(\frac{m}{q}- \frac{1}{m-1}\right)} s^{-\frac{m}{m-1}+\frac{m}{q}} ds\notag \\
&\leq  C  \norm{u_0}_{L^{m-1}(\mathbb{C})} + C  \left(\sup_{t>0}t^{\frac{1}{m-1}-\frac{1}{q}}\norm{v(t)}_{L^{q}(\mathbb{C})}\right)^m{\mathrm B}\left( \gamma,1-\gamma \right),
\end{align}
where $0<\gamma= \frac{m}{q}-\frac{1}{m-1}<1$ for $1<m-1<q<m(m-1)$ and $\mathrm{B}(\cdot,\cdot)$ is the beta function.

The boundedness \eqref{core} and \eqref{BoundLm} leads to the boundedness of RHS of \eqref{le:weak} with norm $\norm{\cdot}_Y$. From this boundedness, we see that, for a given $v\in Y$, the solution operator $\Phi: Y \to Y$ associated with integral equation \eqref{le:weak} that is given by
\begin{equation}
\Phi(v)(t):=u(t) = e^{-t\overline{\Box}_\varphi} u_0 + \int_0^t e^{-(t-s)\overline{\Box}_\varphi}f(s,v(s)) ds,
\end{equation}
is well-defined. Now, we prove that $\Phi$ is a contraction from $\mathbb\mathrm{B}(0,\rho)$ to $\mathbb\mathrm{B}(0,\rho)$, where $\mathbb\mathrm{B}(0,\rho)$ is a ball in Y centered at $0$ with radius $\rho>0$. Indeed, similar to \eqref{core} we have
\begin{align}\label{e:core1}
\norm{\Phi}_Y &\leq  C \norm{u_0}_{L^{m-1}(\mathbb{C})} + \left(\widehat{C} + C \mathrm{B}(\gamma,1-\gamma) \right)\norm{v}_Y^m \cr
&\leq  C \norm{u_0}_{L^{m-1}(\mathbb{C})} + \left(\widehat{C}+ C \mathrm{B}(\gamma,1-\gamma)\right)\rho^m \notag \\
& \leq  \rho,
\end{align}
provided that $\norm{u_0}_{L^{m-1}(\mathbb{C})}$ and $\rho$ are small enough (as above, \(\mathrm{B}\) is the beta function). 

On the other hand, for $v_1,v_2\in \mathbb\mathrm{B}(0,\rho)$ we have
\begin{align}\label{e:co}
&\norm{\Phi(v_1)(t)-\Phi(v_2)(t)}_{L^q(\mathbb{C})} \leq  \int_0^t \norm{e^{-(t-s)\overline{\Box}_\varphi}(f(s,v_1(s)-f(s,v_2(s))))}_{L^q(\mathbb{C})}ds\notag \\
&\qquad \leq  C  \int_0^t (t-s)^{-\frac{m-1}{q}}\left( \norm{v_1(s)}^{m-1}_{L^q(\mathbb{C})}+\norm{v_2(s)}^{m-1}_{L^q(\mathbb{C})}\right)\norm{v_1(s)-v_2(s)}_{L^q(\mathbb{C})}ds.
\end{align}
This leads to
\begin{align}\label{e:core2}
t^{\frac{1}{m-1}-\frac{1}{q}}\norm{\Phi(v_1)-\Phi(v_2)}_{L^q(\mathbb{C})} &\leq  2C \rho^{m-1}\norm{v_1-v_2}_{Y}t^{\frac{1}{m-1}-\frac{1}{q}}\int_0^t (t-s)^{-\frac{m-1}{q}} s^{-\left(\frac{m}{m-1}-\frac{m}{q} \right)}ds \notag \\
&\leq  2\widehat{C}\rho^{m-1}\norm{v_1-v_2}_{Y},
\end{align}
{where $\widehat{C}$ is given as in \eqref{core}.}

Moreover, similarly to \eqref{BoundLm}, we have also that
\begin{align}\label{e:co1}
\|\Phi(v_1)(t) & - \Phi(v_2)(t)\|_{L^{m-1}(\mathbb{C})} \notag \\
&  \leq  \int_0^t \norm{e^{-(t-s)\overline{\Box}_\varphi}(f(s,v_1(s)-f(s,v_2(s))))}_{L^{m-1}(\mathbb{C})}ds\notag \\
&\leq  C  \int_0^t (t-s)^{-\left(\frac{m}{q}-\frac{1}{m-1}\right)}\norm{f(s,v_1(s))-f(s,v_2(s))}_{L^{\frac{q}{m}}(\mathbb{C})}ds\notag \\
&\leq  C  \int_0^t (t-s)^{-\left(\frac{m}{q}-\frac{1}{m-1}\right)}\left( \norm{v_1(s)}^{m-1}_{L^q(\mathbb{C})} + \norm{v_2(s)}^{m-1}_{L^q(\mathbb{C})}  \right) \times \notag \\ 
& \qquad \times \norm{v_1(s)-v_2(s)}_{L^{q}(\mathbb{C})}ds \notag \\ 
&\leq  2C  \rho^{m-1}\norm{v_1-v_2}_Y \mathrm{B}(\gamma,1-\gamma).
\end{align}
Combining inequalities \eqref{e:core2} and \eqref{e:co1}, we get
\begin{equation}\label{e:conti}
\norm{\Phi(v_1)-\Phi(v_2)}_{Y} \leq  2\rho^{m-1}\left(\widehat{C} + C \mathrm{B}(\gamma,1-\gamma)  \right)\norm{v_1-v_2}_Y.
\end{equation}
This shows that $\Phi$ is Lipschitz continuous for small enough $\rho$.

From \eqref{e:core1} and \eqref{e:conti}, we obtain that $\Phi$ is a contraction. By the standard  fixed point argument, $\Phi$ has a fixed point $\hat{u}$ which is also mild solution of \eqref{e:heateq0}. The uniqueness follows from \eqref{e:conti}.

(ii) Since $\norm{u(0)-\hat{u}(0)}_{L^{m-1}(\mathbb{C})}$ and $\norm{\hat{u}(0)}_{L^{m-1}(\mathbb{C})}$ are small enough, the norm $\norm{u(0)}_{L^{m-1}(\mathbb{C})}$ is bounded and can be small enough to guarantee the existence of $u(t)$. Moreover, there exists a positive constant $\widehat{\rho}$ such that $\norm{u}_Y\leq  \widehat{\rho}$.

Now, by the same estimates as \eqref{e:co} and \eqref{e:core2}, we have
\begin{align*}
t^{\frac{1}{m-1} - \frac{1}{q}} & \norm{u-\hat u}_{L^q(\mathbb{C})}
\leq  C \norm{u(0)-\hat{u}(0)}_{L^{m-1}(\mathbb{C})} \\
& \qquad +  t^{\frac{1}{m-1}-\frac{1}{q}}\int_0^t \norm{e^{-(t-s)\overline{\Box}_\varphi}(f(s,u(s)-f(s,\hat{u}(s))))}_{L^q(\mathbb{C})}ds \\
&\leq  C  \norm{u(0)-\hat{u}(0)}_{L^{m-1}(\mathbb{C})} + C   t^{\frac{1}{m-1}-\frac{1}{q}}\int_0^t (t-s)^{-\frac{m-1}{q}} \times \\
& \qquad \times \left( \norm{u(s)}^{m-1}_{L^q(\mathbb{C})}+\norm{\hat{u}(s)}^{m-1}_{L^q(\mathbb{C})}\right)\norm{u(s)-\hat{u}(s)}_{L^q(\mathbb{C})}ds \\
&\leq  C  \norm{u(0)-\hat{u}(0)}_{L^{m-1}(\mathbb{C})}+ C  (\rho^{m-1}+\widehat{\rho}^{\,m-1})t^{\frac{1}{m-1}-\frac{1}{q}} \times \\
& \qquad \times \int_0^t (t-s)^{-\frac{m-1}{q}}s^{-\left( \frac{m}{m-1}-\frac{m}{q} \right)}\left(s^{\left( \frac{m}{m-1}-\frac{m}{q} \right)}\norm{u(s)-\hat{u}(s)}_{L^q(\mathbb{C})}\right)ds.
\end{align*}
{By using the following boundedness estimate (see \eqref{beta}):
\[t^{\frac{1}{m-1}-\frac{1}{q}}\int_0^t (t-s)^{-\frac{m-1}{q}}s^{-\left( \frac{m}{m-1}-\frac{m}{q} \right)}ds = \mathrm{B}\left(\gamma, 1 - \gamma - \nu \right) <\infty,\] where, as above, \(\gamma = \frac{m}{q} - \frac{1}{m-1}\) and \(\nu = \frac{1}{m-1} - \frac{1}{q}\) (which satisfy \(0<\gamma + \nu < 1\)),} and Gronwall's inequality, we get
\begin{equation}
t^{\frac{1}{m-1}-\frac{1}{q}}\norm{u(t)-\hat{u}(t)}_{L^q(\mathbb{C})} \leq  {D}\norm{u(0)-\hat{u}(0)}_{L^{m-1}(\mathbb{C})},
\end{equation}
for some positive constant $D$. This leads to the polynomial stability estimate \eqref{e:pstable}. The proof is complete.

\subsection{Proof of Theorem~\ref{case2}}

(i) As before, for each \(v \in Y\), we consider the following integral equation \begin{equation}\label{le:weak2}
u(t) = e^{-t \overline{\Box}_{\varphi} } u_0 + \int_0^t e^{-(t-s) \overline{\Box}_{\varphi} }f(s,v(s)) ds.
\end{equation}
We prove that 
\begin{equation}\label{e:bounded2}
\norm{u}_{BC(\mathbb{R}_+,L^{n}(\mathbb{C}))}\leq  C\norm{u_0}_{L^{n}(\mathbb{C})} + \overline{C}\norm{v}^m_{BC(\mathbb{R}_+,L^{n}(\mathbb{C}))},
\end{equation}
where $C$ and $\overline{C}$ are  positive constants. Indeed, by using the $L^p$--$L^q$ estimate \eqref{e:lp-lq} (with $p=q=n$ for the first term and $p=n,\, q=\frac{n}{m}$ for the second term of RHS of \eqref{le:weak2}) and  H\"older's inequality, we have
\begin{align*}
\|\hbox{RHS of } & \eqref{le:weak2}\|_{L^n(\mathbb{C})} \leq  \norm{e^{-t \overline{\Box}_{\varphi} } u_0}_{L^{n}(\mathbb{C})}  + \int_0^t \norm{e^{-(t-s) \overline{\Box}_{\varphi} }f(s,v(s))}_{L^{n}(\mathbb{C})} ds \\
&\leq  C \norm{u_0}_{L^{n}(\mathbb{C})} + C\int_0^t (t-s)^{-\left(\frac{m}{n}-\frac{1}{n}\right)}e^{-\widetilde{C}\delta (t-s)}\norm{f(s,v(s))}_{L^{\frac{n}{m}}(\mathbb{C})}ds \cr
&\leq  C\norm{u_0}_{L^{n}(\mathbb{C})} + C\norm{v}^m_{BC(\mathbb{R}_+,L^{n}(\mathbb{C}))}\int_0^t (t-s)^{-\frac{m-1}{n}}e^{- \widetilde{C}\delta (t-s)}ds \cr
&\leq  {C\norm{u_0}_{L^{n}(\mathbb{C})} + C\norm{v}^m_{BC(\mathbb{R}_+,L^{n}(\mathbb{C}))}\int_0^t s^{-\frac{m-1}{n}}e^{- \widetilde{C}\delta s}ds} \cr
&\leq {C\norm{u_0}_{L^{n}(\mathbb{C})} + C\norm{v}^m_{BC(\mathbb{R}_+,L^{n}(\mathbb{C}))}\int_0^\infty s^{-\frac{m-1}{n}}e^{-\widetilde{C}\delta s}ds} \cr
&\leq  C\norm{u_0}_{L^{n}(\mathbb{C})} + C(\widetilde{C}\delta)^{\frac{m-1}{n}}{\Gamma}(1-\theta) \norm{v}^m_{BC(\mathbb{R}_+,L^{n}(\mathbb{C}))} < +\infty,
\end{align*}
where $0<\theta = \frac{m-1}{n}<1$ and ${\Gamma}$ means the gamma function.

By the boundedness of \eqref{e:bounded2}, for each $v\in B(0,\rho)\subset BC(\mathbb{R}_+,L^n(\mathbb{C}))$, we can define a solution operator 
\[\Phi: BC(\mathbb{R}_+,L^n(\mathbb{C})) \to BC(\mathbb{R}_+, L^n(\mathbb{C}))\]
associating with equation \eqref{le:weak} that is given by
\begin{equation}
\Phi(v)(t) := u(t)= e^{-t \overline{\Box}_{\varphi} } u_0 + \int_0^t e^{-(t-s) \overline{\Box}_{\varphi} }f(s,v(s)) ds.
\end{equation}
Below, we prove that $\Phi$ is a contraction from $B(0,\rho)$ to $B(0,\rho)$ for $\rho$ and $\norm{u_0}_{L^n(\mathbb{C})}$ small enough. Then, by the standard fixed point argument, we obtain the unique mild solution of equation \eqref{e:heateq0} in $B(0,\rho)$. In particular, for $v\in B(0,\rho)$ by the same estimates in Assertion (i), we have
\begin{align}\label{core1}
\norm{\Phi(v)}_{BC(\mathbb{R}_+,L^n(\mathbb{C}))} &\leq  C\norm{u_0}_{L^n(\mathbb{C})} + C(\widetilde C\delta)^{\frac{m-1}{n}}{\Gamma}(1-\theta) \norm{v}^m_{BC(\mathbb{R}_+,L^{n}(\mathbb{C}))} \notag\\
&\leq  C\norm{u_0}_{L^n(\mathbb{C})} + C(\widetilde C\delta)^{\frac{m-1}{n}}{\Gamma}(1-\theta) \rho^m \notag \\
& \leq  \rho
\end{align}
provided that $\norm{u_0}_{L^n(\mathbb{C})}$ and $\rho$ are small enough. This shows that $\Phi$ maps $B(0,\rho)$ into $B(0,\rho)$.

Now, for $v_1, v_2\in B(0,\rho)$ we can estimate
\begin{align}\label{core2}
\|\Phi(v_1)& - \Phi(v_2)\|_{BC(\mathbb{R}_+,L^n(\mathbb{C}))} \notag \\
&\leq  \int_0^t \norm{ e^{-(t-s)\overline{\Box}_\varphi}\left( |v_1(s)|^{m-1}v_1(s) - |v_2(s)|^{m-1}v_2(s) \right)}_{L^n(\mathbb{C})} ds \notag \\
&\leq  C\int_0^t (t-s)^{-\frac{m-1}{n}}e^{-\widetilde{C}\delta (t-s)}\norm{|v_1(s)|^{m-1}v_1(s) - |v_2(s)|^{m-1}v_2(s)}_{L^{\frac{n}{m}}(\mathbb{C})}ds \notag \\
&\leq  C\int_0^t (t-s)^{-\left(\frac{m}{n}-\frac{1}{n}\right)}e^{-\widetilde{C}\delta (t-s)}\norm{|v_1(s) - v_2(s)|(|v_1(s)|^{m-1}+v_2(s)|^{m-1})}_{L^{\frac{n}{m}}(\mathbb{C})}ds \notag \\
&\leq  C\int_0^t (t-s)^{-\frac{m-1}{n}}e^{-\widetilde{C}\delta (t-s)}\norm{v_1(s) - v_2(s)}_{L^n(\mathbb{C})}\left(\norm{v_1(s)}_{L^n(\mathbb{C})}^{m-1}+\norm{v_2(s)}_{L^n(\mathbb{C})}^{m-1}\right)ds \notag \\
&\leq  {2C\rho^{m-1}\norm{v_1-v_2}_{BC(\mathbb{R}_+, L^n(\mathbb{C}))}\int_0^t s^{-\frac{m-1}{n}}e^{-\widetilde{C}\delta s}ds}\notag \\
&\leq {2C\rho^{m-1}\norm{v_1-v_2}_{BC(\mathbb{R}_+, L^n(\mathbb{C}))}\int_0^\infty s^{-\frac{m-1}{n}}e^{-\widetilde{C}\delta s}ds}\notag \\
&\leq  2C\rho^{m-1}(\widetilde{C}\delta)^{\frac{m-1}{n}}{\Gamma}(1-\theta)\norm{v_1-v_2}_{BC(\mathbb{R}_+,L^n(\mathbb{C}))}.
\end{align}
Hence,  $\Phi$ is Lipschitz continuous if $\rho$ is small enough. 

Using  \eqref{core1} and \eqref{core2}, we obtain that $\Phi$ is a contraction, provided that $\norm{u_0}_{L^n(\mathbb{C})}$ and $\rho$ are small enough. This shows that there is a fixed point $\hat{u}$ of $\Phi$ (in the small ball $B(0,\rho)$) which is also a mild solution of equation \eqref{e:heateq0}. The uniqueness follows immediately from \eqref{core2}.

(ii) We prove this assertion by using Gronwall inequality. Indeed, since $\norm{\hat{u}(0)}_{L^n(\mathbb{C})}$ and $\norm{u(0)-\hat{u}(0)}_{L^n(\mathbb{C})}$ are small enough, we have $\norm{u(0)}_{L^n(\mathbb{C})}$ is also small enough and by the same way as in Assertion (i), there exists a positive constant $\widetilde{\rho}$ such that we have the boundedness $\norm{u(t)}_{L^n(\mathbb{C})}\leq  \widetilde{\rho}$.

Now, by using the $L^p$--$L^q$ estimate \eqref{e:lp-lq}, we have
\begin{align*}
\norm{u(t)- \hat u(t)}_{L^n(\mathbb{C})} 
&\leq  \norm{e^{-t\overline{\Box}_\varphi}(u(0)-\hat{u}(0))}_{L^{n}(\mathbb{C})} \notag \\
& \quad + \int_0^t \norm{ e^{-(t-s)\overline{\Box}_\varphi}\left( |u(s)|^{m-1}u(s) - |\hat{u}(s)|^{m-1}\hat{u}(s) \right)}_{L^n(\mathbb{C})} ds\cr
&\leq  Ce^{-\widetilde{C}\delta t}\norm{u(0)-\hat{u}(0)}_{L^{n}(\mathbb{C})} + C\int_0^t (t-s)^{-\frac{m-1}{n}}e^{-\widetilde C\delta (t-s)} \times \notag \\ 
& \quad \times \norm{u(s) - \hat{u}(s)}_{L^n (\mathbb{C})} \left(\norm{u(s)}_{L^n(\mathbb{C})}^{m-1}+\norm{\hat{u}(s)}_{L^n(\mathbb{C})}^{m-1}\right)ds \cr
&\leq  Ce^{-\widetilde{C}\delta t}\norm{u(0)-\hat{u}(0)}_{L^{n}(\mathbb{C})} +\cr
& \quad +C\left(\rho^{m-1} + \widetilde{\rho}^{\,m-1}\right)\int_0^t (t-s)^{-\frac{m-1}{n}}e^{-\widetilde{C}\delta (t-s)} \norm{u(s)- \hat{u}(s)}_{L^n(\mathbb{C})}ds.
\end{align*}
Setting $w(t) = e^{\sigma t} \norm{u(t)-\hat{u}(t)}_{L^n(\mathbb{C})}$ with $0<\sigma<\widetilde{C} \delta$, we obtain that
\begin{equation*}
w(t) \leq  C w(0) +C\left(\rho^{m-1} + \widetilde{\rho}^{\,m-1}\right) \int_0^t(t-s)^{-\frac{m-1}{n}}e^{-(\widetilde{C}\delta-\sigma) (t-s)} w(s)ds.
\end{equation*}
Using the fact that $\int_0^t (t-s)^{-\frac{m-1}{n}}e^{-(\widetilde{C}\delta-\sigma)(t-s)}ds$ is bounded by $(\widetilde{C}\delta-\sigma)^{\frac{m-1}{n}}{\Gamma}(1-\theta)$ and Gronwall's inequality we get 
\[ w(s) \leq {\widetilde{D}}w(0). \]
This leads to the asymptotic stability \eqref{e:estable}, as desired. We complete the proof.

\section{Examples}\label{S4}
In this sections, we construct two examples to illustrate the cases \(\delta(\varphi) = 0\) and \(\delta(\varphi)> 0 \). Our examples are motivated by  well-known models for finite and infinite type pseudoconvex real hypersurfaces in \(\mathbb{C}^2\).
\begin{example}
We give an example of subharmonic function \(\varphi\) for which \(\delta(\varphi) > 0\). A typical one is \(\varphi(z) = |z|^2\), but there are more general examples. Indeed, this is the case for any subharmonic \textit{polynomial} \(\varphi\) which is non-harmonic: For such \(\varphi\), there exist \(j_0 \geq 1\) and \(k_0 \geq 1\) such that
\[ 
a^z_{j_0k_0} = \frac{1}{j_0! k_0!}\frac{\partial^{j_0+k_0} \varphi }{\partial z^{j_0} \partial \bar{z}^{k_0}}(z) = c,
 \]
which is a non-zero constant (does not depend on \(z\)).
Thus,
\[ 
	\mu(z,r) =\inf_{j,k\geq 1} \left|\frac{r}{a^z_{jk}}\right|^{1/(j+k)} \leq  \left|\frac{r}{c}\right|^{1/(j_0 + k_0)}.
 \]
Therefore,
\[ 
	\delta(\varphi) = \inf_{\mathbb{C}} \mu(z,1)^{-2} \geq |c|^{2/(j_0+k_0)} > 0,
 \]
as desired. Consequently, Theorem~\ref{case2} with \textit{exponential} stability estimate applies in this case.
\end{example}
\begin{example} Let \(g(t)\) be a smooth convex function on \(\mathbb{R}\) such that \(g(t) = 0 \) for \(t \leq 0\) while \(g(t)\) is increasing for \(t > 0\). For instance, let us take \(h(t) = \exp(-1/t^2)\) on \([0,1/2]\) and, to obtain \(g\), extend it to a convex increasing function on \(\mathbb{R}\) which vanishes on the negative half line.

Let \(\varphi(z) = g(|z|^2)\). By the convexity and monotonicity of \(g\), it follows that \(\varphi(z)\) is subharmonic and non-harmonic. On the other hand, since \(g\) has all derivative vanishing at \(0\), we have that \(a^0_{jk} = 0\) for all \(j , k \geq 1\). Thus, \(\mu(0, r) = + \infty\) for all \(r>0\) and thus \(\delta(\varphi) = 0\). Nevertheless, our Theorem~\ref{case1} still applies: The well-posedness of \eqref{e:heateq0} holds with \textit{polynomial} stability estimate.
\end{example}
 
\end{document}